\documentclass[12pt]{amsart}
\usepackage{bbding}
\usepackage{amssymb}

\usepackage{amssymb, amsthm, hyperref }
\usepackage{enumerate, amsfonts, latexsym,epsfig, color, lpic}
\usepackage{graphicx}

\input xy
\xyoption{all}

\newtheorem {theorem}{Theorem} [section]
\newtheorem {lemma} [theorem] {Lemma}

\newtheorem {corollary} [theorem] {Corollary}

\newtheorem*{mainthm} {Main Theorem}

\newtheorem {remark} [theorem] {Remark}

\newtheorem {convention} [theorem] {Convention}
\newtheorem{Claim}{Claim}

\def\R {\mathbb R}

\def\diam{\rm diam}

\begin{document}

\title[Centralizers of finite subgroups]
{\bf Centralizers of finite subgroups of The Mapping Class Group}

\author[Hao Liang]{Hao Liang}

\address{Hao Liang\\
MSCS UIC 322 SEO, \textsc{M/C} 249\\
851 S. Morgan St.\\
Chicago, IL 60607-7045, USA} \email{hliang8@uic.edu}

\begin{abstract}
In this paper, we study the action of finite subgroups of the
mapping class group of a surface on the curve complex. We prove that
if the diameter of the almost fixed point set of a finite subgroup H
is big enough, then the centralizer of H is infinite.
\end{abstract}

\maketitle

\vspace{4mm}

\section{\large Introduction}

Let $S$ be an orientable surface of finite type with complexity at
least 4, ${\rm Mod}(S)$ be the mapping class group of $S$, $C(S)$ be
the curve complex of $S$ and $\delta$ be the hyperbolicity constant
of $C(S)$. (See Section 4 for the definitions of the above objects, and
references.)
We prove the following theorem.

\begin {mainthm}
Let $H$ be a finite subgroup of ${\rm Mod}(S)$. Let  $C_{H}=\{\nu\in
C(S): {\rm diam}(\it H\cdot\nu)\leq 6\delta\}.$ There
 exists a constant $D$, depending only on the topological type of $S$, such that
if $\rm diam\it (C_{H})\geq D$, then the centralizer
 of $H$ in ${\rm Mod}(S)$ is infinite.
\end {mainthm}

We call points in $C_{H}$ {\em almost fixed points} of $H$. Note
that $C_{H}$ is never empty. In fact, almost fixed points are very
easy to find. Let $\nu\in C(S)$. Then any 1-quasi-centers of the
$H$-orbit of $\nu$ are in $C_{H}$.(See \cite[Lemma III.$\Gamma$.3.3,
p.460]{Bridson-Haefliger} for more detail.)

One of the motivations of the Main Theorem is the following:
Consider a sequence of homomorphisms $\{f_{i}\}$ from a finitely generated group
$G$ to ${\rm Mod}(S)$. This sequence of homomorphisms induce a
sequence of actions of $G$ on $C(S)$. Suppose that the translation
lengths (with respect to some finite generating set of $G$) of these
actions go to infinity. In this case, these actions of $G$ on $C(S)$
converge to a non-trivial action of $G$ on an $\mathbb R$-tree. The
Main Theorem provides some information about this action.

\begin {corollary}
Let $T$ be the $\mathbb R$-tree obtained as above. Let $K$ be the
stabilizer in $G$ of a non-trivial segment in $T$. Then there exists $N$, such that any finite subgroup $H$ of $f_i(K)$ has infinite centralizer in ${\rm Mod}(S)$ for all $i\geq N$.
\end{corollary}

The same phenomenon shows up when one considers the action of a
hyperbolic group on its Cayley graph. We include the proof of the
Main Theorem for hyperbolic group (Theorem 3.1) in this paper for the following
reasons: First, even through experts in geometric group theory might
know the proof for hyperbolic groups, as far as the author knows the
proof is not in the literature. Second, since the two proofs are
similar while the mapping class group case requires many more tools
(such as Masur and Minsky's theory of hierarchies) and is more
technical, we think that the proof of the hyperbolic group case
serves well as a warm-up.

The proofs of both Main Theorems are based on a general fact proved
in Section 2. Consider a ``nice" finitely generated group $G$
admitting a ``nice" action on a infinite metric graph. Lemma 2.1
says if the cardinality of the set of almost fixed points (see
Section 2 for definition) of a finite subgroup is big enough, then
the centralizer of the finite subgroup is infinite.

In Section 3, we use the hyperbolicity of the Cayley graph of a
hyperbolic group to show that having two almost fixed points being
far apart implies having a lot of points with small $H$-orbit. This
is Lemma 3.2. Then we show that the action in this case is ``nice''
in the sense of Lemma 2.1 and main theorem for hyperbolic groups
(Theorem 3.1) follows. In Section 4, we introduce the basic definitions we need to state
the main theorem. In Section 5, we prove the Main Theorem for the mapping class group.
The proof of main theorem for the mapping class group relies heavily
on the theory of hierarchies. Readers who are not familiar with the theory of hierarchies should
read \cite{Masur-Minsky2}. In Section 6, we prove Corollary 1.2.

The author is grateful to Daniel Groves, who has taught the author a
lot about the interplay between the theory of hyperbolic group and
the theory of mapping class group through hierarchies and without
whose many helpful suggestions, this paper would not have been
possible.

\section{\large The key lemma}

Lemma 2.1 is a key fact we need in the proofs of the main theorems,
in both the hyperbolic case and the mapping class group case.

In order to state Lemma 2.1, we need to introduce some notation.
Consider a finitely generated group $G$ acting properly and
cocompactly on a infinite locally finite metric graph $K$ by
isometries. Let $H$ be a finite subgroup of $G$. Let $a$ be a
positive integer.

Suppose the cardinalities of finite subgroups of $G$ are bounded
above by some number $C_{0}$.

Let $K^{(0)}$ be the set of vertices of $K$ and $C_{1}$ be the
number of points in $K^{(0)}/G$.

For $p\in K$, let $B(p,a)$ denote the $a$-neighborhood of $p$ in $K$
and $\rm card_{\it v}\it(B(p,a))$ be the number of vertices in
$B(p,a)$. Let $C_{2}$ be an upper bound for $\{{\rm
card}_{v}(B(p,a)):p\in K^{(0)}\}$.

Let $C_{3}=\rm Max\{\rm card(stab(\it p)):p\in K^{(0)}\}$, where
$\rm stab(\it p)$ is the stabilizer of $p$ in $G$.

\begin {lemma}
Let $P_{H}=\{p\in K^{(0)}:\diam\it(H\cdot p)\leq a\}$. Then there
exists a constant $N$, depending only on $C_{0},C_{1},C_{2},C_{3}$,
such that if $\rm card\it(P_{H})\geq N$, the centralizer of H in G
is infinite.
\end{lemma}

\begin{proof}
It suffices to take
$N=((C_{0}+1)(C_{3})^{C_{0}}+1)C_{1}(C_{2})^{C_{0}}$. Assume $\rm
card\it(P_{H})\geq N$. We show that in this case the centralizer of
$H$ is infinite.

By definition, $C_{1}$ is the number of $G$-orbits in $K^{(0)}$. By
the pigeonhole principle, there are at least
\begin{eqnarray*} r_{1}=\frac{N}{C_{1}}
\end{eqnarray*}
points of $P_{H}$ in the same orbit. Choose a subset $P=\{p_{1},
\cdot\cdot\cdot, p_{r_{1}}\}$ of $P_{H}$ so that all elements of $P$
are in the same G-orbit. Choose $g_{i}\in G$ so that $g_{i}\cdot
p_{1}=p_{i}$ for $2\leq i\leq r_{1}$. Note that $g_{i}^{-1}$ induces
an isometry from $B(p_{i},a)$ to $B(p_{1},a)$.

Let $H=\{h_{1}, \cdot\cdot\cdot, h_{d}\}$. First, we consider the
action of $h_{1}$. For any $p_{i}\in P$, we have $h_{1}\cdot
p_{i}\in B(p_{i}, a)$ by the definition of $P_{H}$. Therefore,
$g_{i}^{-1}\cdot h_{1}\cdot p_{i}\in B(p_{1},a)$. Since $\rm
card_{\it v}\it(B(p_{1},a))\leq \it C_{\rm 2}$, by the pigeonhole
principle, there exists $v_{1}\in B(p_{1},a)$ such that for at least
$\frac{r_{1}}{C_{2}}$ many $i$, $g_{i}^{-1}\cdot h_{1}\cdot
p_{i}=v_{1}$. Let $I_{\rm 1}$ be the subset of $\{1,
\cdot\cdot\cdot, r_{1}\}$ such that for any $i\in I_{1}$, we have
$g_{i}^{-1}\cdot h_{1}\cdot p_{i}=v_{1}$, which is equivalent to
$h_{1}\cdot p_{i}=g_{i}\cdot v_{1}$.

Now consider $h_{2}$. As above, by the pigeonhole principle, there
exists $v_{2}\in B(p_{1},a)$, and a subset $I_{2}$ of $I_{1}$ with
$\rm card\it(I_{\rm 2})\geq \frac{r_{\rm 1}}{(C_{\rm 2})^{\rm 2}}$,
such that $h_{2}\cdot p_{i}=g_{i}\cdot v_{2}$ for all $i\in I_{2}$.

Repeat this process for all the elements of H, we have
\begin{eqnarray*} h_{t}\cdot p_{i}=g_{i}\cdot v_{t}
\end{eqnarray*}
for all $1\leq t\leq d$ and all $i\in I_{d}$, where $I_{d}\subset
I_{d-1}\subset\cdot\cdot\cdot\subset I_{1}$ and
\begin{eqnarray*}
r_{2}=\rm card\it(I_{d})\geq\frac{r_{1}}{(C_{\rm 2})^{d}}.
\end{eqnarray*}

Fix an element $b\in I_{d}$. For any $i\in I_{d}$, we have:

\begin{eqnarray*}h_{1}\cdot g_{i}\cdot g_{b}^{-1}\cdot p_{b}&=&h_{1}\cdot g_{i}\cdot
p_{1}\\&=&h_{1}\cdot p_{i}\\&=&g_{i}\cdot v_{1}\\&=&g_{i}\cdot
g_{b}^{-1}\cdot h_{1}\cdot p_{b}
\end{eqnarray*}
Therefore we have:
\begin{eqnarray*}
h_{1}^{-1}\cdot g_{b}\cdot g_{i}^{-1}\cdot h_{1}\cdot g_{i}\cdot
g_{b}^{-1}\in \rm stab\it(p_{b})
\end{eqnarray*}
We know that $\rm card(stab(\it p_{b}))\leq \it C_{3}$. Now apply
the pigeonhole principle again, we know that there exists a subset
$I_{d}^{1}$ of $I_{d}$ with $\rm
card\it(I_{d}^{1})\geq\frac{r_{2}-1}{C_{3}}$, such that for any $i,
j\in I_{d}^{1}$,
\begin{eqnarray*}
h_{1}^{-1}\cdot g_{b}\cdot g_{i}^{-1}\cdot h_{1}\cdot g_{i}\cdot
g_{b}^{-1}=h_{1}^{-1}\cdot g_{b}\cdot g_{j}^{-1}\cdot h_{1}\cdot
g_{j}\cdot g_{b}^{-1},
\end{eqnarray*}
which is equivalent to:
\begin{eqnarray*}
g_{j}\cdot g_{i}^{-1}\cdot h_{1}=h_{1}\cdot g_{j}\cdot g_{i}^{-1}.
\end{eqnarray*}
Repeat this process for all the elements of H, we get a subset
$I_{d}^{d}$ of $I_{d}$, with $\rm card\it(I_{d}^{d})\geq\frac{r_{\rm
2}-1}{(C_{\rm 3})^{d}}$, such that for any $i, j\in I_{d}^{d}$, any
$1\leq t\leq d$,
\begin{eqnarray*}
g_{j}\cdot g_{i}^{-1}\cdot h_{t}=h_{t}\cdot g_{j}\cdot g_{i}^{-1}.
\end{eqnarray*}
Fix $c\in I_{d}^{d}$. Then for all $i\in I_{d}^{d}$, all $h_{t}\in
H$, we have:
\begin{eqnarray*}
g_{c}\cdot g_{i}^{-1}\cdot h_{t}=h_{t}\cdot g_{c}\cdot g_{i}^{-1}
\end{eqnarray*}
Hence $g_{c}\cdot g_{i}^{-1}$ centralizes $H$ for all $i\in
I_{d}^{d}$. Therefore, there are at least $\rm card\it(I_{d}^{d})$
elements in the centralizer of H. But since
$N=((C_{0}+1)(C_{3})^{C_{0}}+1)C_{1}(C_{2})^{C_{0}}$, we have:
\begin{eqnarray*}
r_{1}=\frac{N}{C_{1}}=((C_{0}+1)(C_{3})^{C_{0}}+1)(C_{2})^{C_{0}}.
\end{eqnarray*}
Therefore, since $d\leq C_{0}$, we have:
\begin{eqnarray*}
r_{2}\geq\frac{r_{1}}{(C_{2})^{d}}\geq(C_{0}+1)(C_{3})^{C_{0}}+1.
\end{eqnarray*}
So, again using the fact that $d\leq C_{0}$, we have:
\begin{eqnarray*}
{\rm card}(I_{d}^{d})\geq\frac{r_{2}-1}{(C_{3})^{d}} \geq C_{0}+1
\end{eqnarray*}
So there are at least $C_{0}+1$ elements in the centralizer of H,
but any finite subgroup of $G$ has cardinality at most $C_{0}$, so
the centralizer of H must be infinite.
\end{proof}

\section {\large Main theorem and Proof - the hyperbolic group case }

We use the convention that a {\em $\delta$-hyperbolic space} is a geodesic metric
space in which all geodesics triangles are $\delta$-thin. (See \cite[Definition III.H 1.16,
p.408]{Bridson-Haefliger} for more detail.)

\begin {theorem}\label{t:main}
Let $G$ be a hyperbolic group with
$\{g_{1},\cdot\cdot\cdot,g_{n}\}$ as an generating set. Let $K_{G}$ be the Cayley graph of G with
respect to the given generating set. Let $\delta$ be the hyperbolicity constant for $K_{G}$. Let $H$ be a
finite subgroup of $G$. Let
\[X_{H}=\{x\in K_{G}:\diam(\it H\cdot x)\leq 6\delta\}.\] There
 exists a constant $D$, depending only on $\delta$ and $n$, such that
if $\diam\it (X_{H})\geq D$, then the centralizer
 of $H$ in $G$ is infinite.
\end {theorem}

We call $x\in X_{H}$ {\em almost fixed points} of $H$.

\begin {lemma}
Let $x, y\in X_{H}$. Suppose $d(x,y)\geq 20\delta$. Let $[x,y]$ be a
geodesic in $K_{G}$ connecting $x$ and $y$. Then for any vertex
$z\in [x,y]$ such that $d(x,z)\geq 6\delta+1$ and $d(z,y)\geq
6\delta+1$, we have ${\rm diam}(H\cdot z)\leq 8\delta$.
\end{lemma}

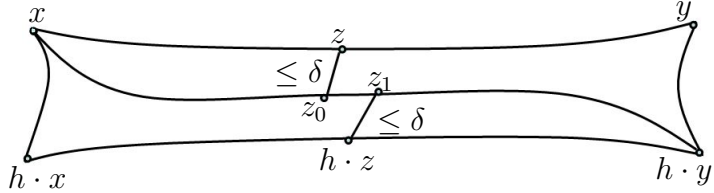
\begin{lpic}[l(12mm)]{good1(0.4, 0.4)}
\lbl[b]{5,70;$x$}

\lbl[b]{220,70;$y$}

\lbl[t]{220,25;$h\cdot y$}

\lbl[t]{5,22;$h\cdot x$}

\lbl[b]{105,63;$z$}

\lbl[t]{108,28;$h\cdot z$}

\lbl[t]{97,43;$z_{0}$}

\lbl[r]{100,53;$\leq \delta$}

\lbl[l]{118,37;$\leq \delta$}

\lbl[b]{119,47;$z_{1}$}

\end{lpic}

\begin{proof}
It suffices to prove that $d(h\cdot z, z)\leq 8\delta$ for all $h\in
H$.

Consider the geodesic triangle with edges: \[[x,y], [x,h\cdot y],
[y,h\cdot y].\] $K_{G}$ is $\delta$-hyperbolic, so the triangle
satisfies the thin triangle condition. Since $d(z,y)\geq 6\delta+1$
and $d(y,h\cdot y)\leq 6\delta$, there is a point $z_{0}\in
[x,h\cdot y]$ such that $d(z, z_{0})\leq\delta$ and
$d(x,z_{0})=d(x,z)$.

Now consider the triangle with edges \[[x,h\cdot x], [x,h\cdot y],
[h\cdot x,h\cdot y]=h\cdot[x,y].\] As above, since $d(h\cdot
z,h\cdot x)=d(x,z)\geq 6\delta+1$ and $d(x,h\cdot x)\leq 6\delta$,
there is a point $z_{1}\in [x,h\cdot y]$ such that $d(h\cdot z,
z_{1})\leq\delta$ and $d(h\cdot y,z_{1})=d(h\cdot y,h\cdot z)$. So
we have:
\begin{eqnarray*} d(z_{0},z_{1})&=& \mid
d(x,z_{0})+d(h\cdot y,z_{1})-d(x,h\cdot y)\mid\\&=&\mid
d(x,z)+d(h\cdot y,h\cdot z)-d(x,h\cdot y)\mid\\&=&\mid d(h\cdot
x,h\cdot z)+d(h\cdot y,h\cdot z)-d(x,h\cdot y)\mid\\&=&\mid d(h\cdot
x,h\cdot y)-d(x,h\cdot y)\mid\\&\leq& 6\delta.
\end{eqnarray*}

Now we know: $d(h\cdot z, z)\leq d(z,z_{0})+ d(h\cdot z,z_{1})+
d(z_{0},z_{1})\leq \delta+\delta+6\delta=8\delta$.
\end{proof}

Apply Lemma 2.1 to the action of $G$ on $K_{G}$, we get the
following lemma.

\begin {lemma}
Let $H$ and $G$ be as in Theorem 3.1. Let $P_{H}=\{x\in
K_{G}:\diam\it(H\cdot x)\leq 8\delta\}$. There exists a constant
$N$, depending only on $\delta$ and $n$, such that if $\rm
card\it(P_{H})\geq N$, then the centralizer of $H$ in $G$ is
infinite.
\end{lemma}

\begin{proof}
In order to apply Lemma 2.1, it suffices to show that in the current
situation, $C_{0},C_{1},C_{2},C_{3}$ are finite and they depend only
on $\delta$ and $n$.

By \cite[Theorem III.$\Gamma$.3.2, p.459]{Bridson-Haefliger}, there
exists an upper bound, depending only on $\delta$ and $n$, for the
cardinality of finite subgroups of $G$. So $C_{0}$ is finite and
depends only on $\delta$ and $n$. We have $C_{1}=1$ since $K_{G}/G$
has only one vertex. Also $C_{2}$ is finite and depends only on
$\delta$ and $n$ by the definition of Cayley graph. Finally,
$C_{3}=1$ since the action is free.
\end{proof}

\begin{proof}[Proof of Theorem 3.1]
Let $D=N+12\delta+4$, where N is the constant given by the previous
lemma. Then $D$ depends only on $\delta$ and $n$. Let $x, y\in
X_{H}$ such that $d(x,y)\geq D$. Let $[x,y]$ be a geodesic
connecting $x,y$. Let $B=\{z\in [x,y]\mid d(z,x)\geq 6\delta+1,
d(z,y)\geq 6\delta+1\}$. Then $\rm card\it(B)\geq N$ and $B\subset
P_{H}$, where $P_{H}$ is as in the statement of Lemma 3.3. So $\rm
card\it(P_{H})\geq N$. Therefore, by Lemma 3.3, the centralizer of
$H$ in $G$ is infinite.
\end{proof}

\section { {\rm Mod}(S): \large basic definitions}

Let $S=S_{\gamma, p}$ be an orientable surface of finite type, with
genus $\gamma$ and $p$ punctures. The only surfaces with boundary we
consider will be subsurfaces of $S$. The {\em complexity} of $S$ is
measured by $\xi(S)=3\gamma(S)+p(S)$. In this paper, we only
consider surfaces with $\xi\geq 4$. The only exception is the
annulus, which will only appear as a subsurface of $S$.

The {\em Mapping Class Group} of S, denoted  by ${\rm Mod}(S)$, is
the group of orientation-preserving homeomorphisms of $S$ modulo
isotopy.

A {\em curve} on S will always mean the isotopy class of a simple
closed curve, which is not null-homotopic or homotopic into a
puncture.

For surface $S$ with $\xi\geq 5$, the {\em graph of curves} $C(S)$
consists of a vertex for every curve, with edges joining pairs of
distinct curves that have disjoint representatives on $S$. The graph
of curves is the 1-skeleton of the curve complex introduced by
Harvey.

When $\xi= 4$, the surface $S$ is either a once-punctured torus
$S_{1,1}$ or four times punctured sphere $S_{0,4}$. We have an
alternate definition for the graph of curves $C(S)$: Vertices are
still curves. Edges are given by pairs of distinct curves that have
representatives that intersect once (for $S_{1,1}$) or twice (for
$S_{0,4}$).

By assigning length $1$ to each edge we make $C(S)$ into a metric
graph. We use $d_{S}$ to denote this metric.  Masur and Minsky prove the following theorem (\cite[Theorem 1.1]{Masur-Minsky}).

\begin{theorem}
$C(S)$ is an
$\delta$-hyperbolic metric space, where $\delta$ depends on $S$. Except when $S$ is a sphere with 3 or fewer punctures, $C(S)$ has infinite diameter.
\end{theorem}

Since elements in ${\rm Mod}(S)$ preserve disjointness of curves,
${\rm Mod}(S)$ acts on $C(S)$ by isometry. This action is cocompact
since there are only finitely many curves on $S$ up to
homeomorphisms, but it is far from proper.

\begin{convention}
For the rest of the paper, by an element $x\in C(S)$ we always
mean a vertex of $C(S)$ and similarly for a subset of $C(S)$.
\end{convention}

\section {\large Main theorem and Proof }

In this section we prove the Main Theorem for ${\rm Mod}(S)$.  First, recall the statement.

\begin {theorem}[Main]
Let $H$ be a finite subgroup of ${\rm Mod}(S)$. Let
\[C_{H}=\{\nu\in C(S):\diam(\it H\cdot \nu)\leq 6\delta\}.\] There
exists a constant $D$, depending only on the topological type of
$S$, such that if $\diam\it (C_{H})\geq D$, then the centralizer of
$H$ in ${\rm Mod}(S)$ is infinite.
\end{theorem}

\begin{proof}
Just as in the hyperbolic groups case, we first show that having two
almost fixed points being far apart implies having a lot of points
with small $H$-orbit. The idea of the following lemma is the same as
Lemma 3.2.

\begin {lemma}
Let $\nu_{0}, \nu_{1}\in C_{H}$. Suppose $d(\nu_{0},\nu_{1})\geq
20\delta$. Let $[\nu_{0},\nu_{1}]$ be a geodesic in $C(S)$
connecting $\nu_{0}$ and $\nu_{1}$. Then for any vertex $b\in
[\nu_{0},\nu_{1}]$ such that $d_{S}(\nu_{0},b)\geq 6\delta+1$ and
$d_{S}(b,\nu_{1})\geq 6\delta+1$, we have $\diam_{\it S}\it(H\cdot
b)\leq 8\delta$.
\end{lemma}

If we can apply Lemma 2.1 to prove a similar result as Lemma 3.3 for
the action of ${\rm Mod}(S)$ on $C(S)$, the Main Theorem will
follow. But one immediately sees that such result can not be proved
in the same way for two reasons: $C(S)$ is locally infinite and
action of ${\rm Mod}(S)$ on $C(S)$ has infinite vertex-stabilizers.
However, we can prove a similar result for a ``nicer'' action of
${\rm Mod}(S)$ on a locally finite graph.

Let $\mathcal{M}(S)$ be the graph of complete clean markings of the
surface $S$ as defined by Masur and Minsky in \cite[section
7.1]{Masur-Minsky2}. We use $d_{\mathcal{M}}$ to denote the metric
on $\mathcal{M}(S)$. Recall that $\mathcal{M}(S)$ is locally finite
and admits an proper and cocompact action by ${\rm Mod}(S)$ by
isometries. Apply Lemma 2.1 to the action of ${\rm Mod}(S)$ on
$\mathcal{M}(S)$. We get the following lemma.

\begin {lemma}
Let $a$ be any positive integer. Let $H$ be a finite subgroup of
${\rm Mod}(S)$. Let $P_{H}^{a}=\{P\in \mathcal{M}(S):{\rm
diam}(H\cdot P)\leq a\}$. There exists a constant $N$, depending
only on $S$ and $a$, such that if $\rm card\it(P_{H}^{a})\geq N$,
the centralizer of $H$ is infinite.
\end{lemma}

\begin{proof}
In order to apply Lemma 2.1, it suffices to show that in the current
situation, $C_{0},C_{1},C_{2},C_{3}$ are finite and they depend only
on $S$ and $a$.

By Nielsen Realization Theorem (See \cite{WOLPERT} for a proof for
the case of puncture surfaces) every finite subgroup of ${\rm Mod}(S)$ can be realized as a subgroup of the isometry group of the surface with some hyperbolic structure.
By Hurwitz's automorphisms theorem, the size of the isometry group of a punctured hyperbolic surface is
bounded above. (The bound is $84(g-1)$ when $g\geq 2$. When $g\leq 1$, a similar
argument as in \cite[Section 7.2]{FM} gives an upper bound for the
size of the isometry group.) Hence the orders of finite subgroups of ${\rm
Mod}(S)$ are bounded above by a constant which depends only on the
topological type of S. So $C_{0}$ is finite and depends only on $S$.
By the construction of $\mathcal{M}(S)$, both $C_{1}$ and $C_{3}$
are finite and depend only on $S$. For the same reason, $C_{2}$ is
finite and depends only on $S$ and $a$.
\end{proof}

Lemma 5.2 and Lemma 5.3 together do not give the result we want
since they are about actions of ${\rm Mod}(S)$ on different metric
spaces. In order to connect these two actions, we use Masur and
Minsky's theory of hierarchies.

Let $\nu_{0}$, $\nu_{1}$ be $C_{H}$. Let $\mu_{0}$, $\mu_{1}$ be
markings (\cite[section 2.5]{Masur-Minsky2}) such that $\nu_{0}\in
\rm base\it(\mu_{0})$, $\nu_{1}\in \rm base\it(\mu_{1})$. Let
$\mathcal{H}=[\mu_{0}, \mu_{1}]$ be a hierarchy (\cite[Definition
4.4]{Masur-Minsky2}) with initial marking $\mu_{0}$, terminal
marking $\mu_{1}$ and with the main geodesic connecting $\nu_{0}$,
$\nu_{1}$. For $h\in H$, Let $\mathcal{H}_{h}$ be the $h$ translate
of $\mathcal{H}$.

Let $B$ be the set of vertices in [$\nu_{0},\nu_{H}$], the main
geodesic of $\mathcal{H}$, such that $d_{S}(\nu_{0},b)\geq
14\delta+5$ and $d_{S}(b,\nu_{1})\geq 14\delta+5$. For any $b\in B$,
$h\in H$, let $\mu_{b}$ be a marking compatible with a slice
(\cite[section 5]{Masur-Minsky2}) of $\mathcal{H}$ at $b$. Then
$h\cdot \mu_{b}$ is a marking compatible with a slice of
$\mathcal{H}_{h}$ at $h\cdot b$. Let
$\mathcal{H}_{b}^{h}=[\mu_{b},h\cdot \mu_{b}]$ be a hierarchy
connecting $\mu_{b}$ and $h\cdot \mu_{b}$.

\begin {lemma}
$\mathcal{H}_{b}^{h}$ is $(K, M')$-pseudo-parallel (\cite[Definition
6.5]{Masur-Minsky2}) to $\mathcal{H}$ where $K$ and $M'$ depend only
on $S$.
\end{lemma}

\begin{proof}
By Lemma 5.2, the main geodesic [$\nu_{0},\nu_{H}$] of $\mathcal{H}$
and the main geodesic $h\cdot [\nu_{0},\nu_{H}]$ of
$\mathcal{H}_{h}$ are $(8\delta+2, 2\delta +1)$-parallel
(\cite[Definition 6.4]{Masur-Minsky2}) at $b$ and $h\cdot b$ for all
$b\in B$ and $h\in H$. Now apply \cite[Lemma 6.7]{Masur-Minsky2}.
\end{proof}

Before we can define the constant $D$ in the Main Theorem, we need
the following lemma.

\begin {lemma}
Let $\mathcal{H}$ be a hierarchy. Let $c$ be any positive number.
Suppose that the lengths of all the geodesics in $\mathcal{H}$ are
less than $c$. Then the distance between the initial marking and the
terminal marking of $\mathcal{H}$ in $\mathcal{M}(S)$ is less than
$d$, where $d$ is a number depending only on $c$ and the topological
type of $S$.
\end {lemma}

\begin{proof}
Apply \cite[Theorem 6.12]{Masur-Minsky2} with $M=c$.
\end{proof}

Let $M$ be the constant in \cite[Theorem 3.1]{Masur-Minsky2}. Let
$M_{1}$, $M_{2}$ be the constants in \cite[Lemma
6.2]{Masur-Minsky2}. Let $K$ and $M'$ be the constants in Lemma 5.4.
Let $e=2M+8M_{1}+M_{2}+2K+M'$. Let $d$ be the constant given by
Lemma 5.5 with the above $c=e+2M_{1}$. Let $N$ be the constant given
by Lemma 5.3 with $a=d$. Let $D=N+12\delta+10$. We will show this is
the constant $D$ we want. Note that $D$ depends only on the
topological type of $S$.

The rest of the proof is devoted to showing that the centralizer of
$H$ is infinite provided that $d_{S}(\nu_{0}, \nu_{1})\geq D$.

\vspace{4mm}

The proof will break into 2 cases: If the length of the hierarchies
$H_{b}^{h}$ are bounded for all $b\in B$, $h\in H$, then the
distance between $\mu_{b}$ and $h\cdot \mu_{b}$ in $\mathcal{M}(S)$
are bounded. In this case, we have enough almost-fixed points in
$\mathcal{M}(S)$ and we can apply Lemma 5.3 to conclude that the
centralizer of $H$ in ${\rm Mod}(S)$ is infinite. On the other hand,
if there is a ``long" hierarchy $H_{b}^{h}$, we are able to use an
argument in Jing Tao's thesis \cite{Tao} to show that there exists a
subsurface $Y$ of $S$ such that elements of $H$ either preserve $Y$
or take $Y$ completely off itself. Using this fact we construct an
infinite order element of ${\rm Mod}(S)$ which centralizes $H$.

\vspace{4mm}

{\bf Case 1}:  For any $b\in B$, $h\in H$ and any subsurface $Y$ of
$S$ supporting a geodesic of $H_{b}^{h}$, $d_{Y}(\mu_{b},h\cdot
\mu_{b})\leq e$.(See \cite[section 2.3]{Masur-Minsky2} for the
definition of $d_{Y}$.)

\begin {Claim}\label{t:kobe}
In Case 1, $d_{\mathcal{M}}(\mu_{b},h\cdot\mu_{b})\leq d$ for all
$b\in B$, $h\in H$, where $d$ is one of the numbers we used to
define $D$.
\end {Claim}

\begin{proof}
By \cite[Lemma 6.2]{Masur-Minsky2}, the geodesic in $Y$ has length
at most $e+2M_{1}$. Now the claim follows from Lemma 5.5 and the
definition of $d$.
\end{proof}

Note that Claim 1 says that for any $b\in B$, $\mu_{b}$ is in
$P_{H}^{d}$. Since $d_{S}(\nu_{0},\nu_{1})\geq D$, we have $\mid
P_{H}^{d}\mid\geq\mid B\mid\geq D-12\delta-8\geq N$. By Lemma 5.3
and the definition of $N$, the centralizer of $H$ is infinite and
the proof is complete in Case 1.

\vspace{5mm}

{\bf Case 2}: There exists $b_{l}\in B$, $h_{l}\in H$, and a
subsurface $Y$ of $S$ which supports a geodesic of ${\mathcal
H}_{b_{l}}^{h_{l}}$, such that $d_{Y}(\mu_{b_{l}},h_{l}\cdot
\mu_{b_{l}})\geq e$.

\begin {lemma}\label{t:kobe1}
In Case 2, $d_{Y}(\mu_{0},\mu_{1})\geq 2M+4M_{1}+M_{2}$.
\end {lemma}

\begin{proof}
Since we are in Case 2 we have $d_{Y}(\mu_{b_{l}},h_{l}\cdot
\mu_{b_{l}})\geq e\geq M_{2}$. So by \cite[Lemma 6.2]{Masur-Minsky2}
$Y$ supports a geodesic of $\mathcal{H}_{b_{l}}^{h_{l}}$ of length
at least $e-2M_{1}=2M+6M_{1}+M_{2}+2K+M'$. In particular, this
geodesic has length bigger than $M'$. By Lemma 5.4,
$\mathcal{H}_{b_{l}}^{h_{l}}$ is $(K, M')$-pseudo-parallel to
$\mathcal{H}$. So $Y$ also supports a geodesic of $\mathcal{H}$,
whose length is at least
$2M+6M_{1}+M_{2}+2K+M'-2K=2M+6M_{1}+M_{2}+M'$. Now apply \cite[Lemma
6.2]{Masur-Minsky2} again, we know that $d_{Y}(\mu_{0}, \mu_{1})\geq
2M+6M_{1}+M_{2}+M'-2M_{1}\geq 2M+4M_{1}+M_{2}$ as we claim.
\end{proof}

\begin {lemma}
Let $b\in B$, $h\in H$. Suppose $Y$ supports a geodesic of
${\mathcal H}_{b}^{h}$. Then $d_{Y}(\mu_{0}, h\cdot\mu_{0})\leq M$
and $d_{Y}(\mu_{1}, h\cdot\mu_{1})\leq M$.
\end{lemma}

\begin{proof}
Let $[\nu_{b},h\cdot\nu_{b}]$ be the main geodesic in ${\mathcal
H}_{b}^{h}$. Since $Y$ supports a geodesic in ${\mathcal
H}_{b}^{h}$, $Y$ must be forward subordinate (See \cite[section
4.1]{Masur-Minsky2} for definition.) to $[\nu_{b},h\cdot\nu_{b}]$ at
some vertex $\nu$. Let $l$ be any boundary component of Y. Then
$d_{S}(l,\nu)=1$. Since $\nu_{0}\in C_{H}$, we have
 $d_{S}(\nu_{0},h\cdot\nu_{0})\leq 6\delta$. Let
$[\nu_{0},h\cdot\nu_{0}]$ be a geodesic connecting
$\nu_{0},h\cdot\nu_{0}$. Let $\nu_{i}$ be a point on
$[\nu_{0},h\cdot\nu_{0}]$.  By the triangle inequality,
\begin{eqnarray*}
d_{S}(\nu,\nu_{i})&\geq&
d_{S}(\nu_{0},\nu_{b})-d_{S}(\nu,\nu_{b})-d_{S}(\nu_{i},\nu_{0})\\&\geq&
d_{S}(\nu_{0},\nu_{b})-d_{S}(\nu_{b},h\cdot\nu_{b})-d_{S}(\nu_{0},h\cdot\nu_{0})\\&\geq&
(14\delta+5)-(8\delta+2)-6\delta\\&=& 3.
\end{eqnarray*}
Then $d_{S}(l,\nu_{i})\geq d_{S}(\nu,\nu_{i})-d_{S}(l,\nu)\geq
3-1=2$. Therefore $\nu_{i}$ intersects $l$. As a result, $\nu_{i}$
intersects $Y$. And this is true for all
$\nu\in[\nu_{0},h'\cdot\nu_{0}]$. By \cite[Theorem
3.1]{Masur-Minsky2}, $d_{Y}(h\cdot\nu_{0}, \nu_{0})\leq M$. An exact
same argument shows $d_{Y}(\mu_{1}, h\cdot\mu_{1})\leq M$.
\end{proof}

We prove the following key lemma for Case 2 using an argument in
\cite[Lemma 3.3.4]{Tao}.

\begin {lemma}
In Case 2, for any $h\in H$, either $h(Y)=Y$ or $h(Y)$ and $Y$ are
disjoint.
\end{lemma}

\begin{proof}
Let $h\in H$. Applying Lemma 5.6 and Lemma 5.7, we have
\begin{eqnarray*}
d_{h^{-1}(Y)}(\mu_{0},\mu_{1})&=& d_{Y}(h\cdot\mu_{0},
h\cdot\mu_{1})\\&\geq&
d_{Y}(\mu_{0},\mu_{1})-d_{Y}(\mu_{0},h\cdot\mu_{0})-d_{Y}(\mu_{1},h\cdot\mu_{1})\\&\geq&
2M+4M_{1}+M_{2}-M-M\\&=& 4M_{1}+M_{2}\\&\geq&M_{2}.
\end{eqnarray*}
So by \cite[Lemma 6.2]{Masur-Minsky2}, $h^{-1}(Y)$ is also a domain
in $\mathcal{H}$. Suppose $h^{-1}(Y)\neq Y$. Then since $h^{-1}(Y)$
and $Y$ have the same complexity, they are either disjoint from each
other or they interlock(i.e. intersect but do not contain each
other).

Suppose $h^{-1}(Y)$ and $Y$ are not disjoint. Then by \cite[Lemma
4.18]{Masur-Minsky2}, $h^{-1}(Y)$ and $Y$ are time-ordered
(\cite[Definition 4.16]{Masur-Minsky2}).

First suppose $Y\prec_{t} h^{-1}(Y)$(Here $\prec_{t}$ is the
notation for time order). As in the proof of \cite[Lemma
6.11]{Masur-Minsky2}, there exist a slice in $\mathcal{H}$ so that
its associated compatible marking $\nu$ satisfies
\begin{eqnarray*}
d_{Y}(\nu, \mu_{1})\leq M_{1}\hspace{4mm} {\rm and}
\hspace{4mm}d_{h^{-1}(Y)}(\nu, \mu_{0})\leq M_{1}.
\end{eqnarray*}
Then since $d_{h^{-1}(Y)}(\nu, \mu_{0})=d_{Y}(h\cdot\mu_{0},
h\cdot\nu)$, we have
\begin{eqnarray*}
d_{Y}(\mu_{0},h\cdot\nu)&\leq& d_{Y}(\mu_{0},
h\cdot\mu_{0})+d_{Y}(h\cdot\mu_{0}, h\cdot\nu)\\&\leq& M+M_{1}.
\end{eqnarray*}
By Lemma 5.6, we have
\begin{eqnarray*}
d_{Y}(\mu_{1},h\cdot\nu)&\geq& d_{Y}(\mu_{0},
\mu_{1})-d_{Y}(\mu_{0}, h\cdot\nu)\\&\geq&2M+4M_{1}+M_{2}-
(M+M_{1})\geq 2M_{1}.
\end{eqnarray*}
Therefore, by \cite[Lemma 1]{BM}, we have
\begin{eqnarray*}
d_{h^{-1}(Y)}(\mu_{0},h\cdot\nu)&\leq& 2M_{1}.
\end{eqnarray*}
Hence we get
\begin{eqnarray*}
d_{Y}(\mu_{0},h^{2}\cdot\nu)&\leq& d_{Y}(\mu_{0},
h\cdot\mu_{0})+d_{Y}(h\cdot\mu_{0}, h^{2}\cdot\nu)\\&\leq&
M+d_{h^{-1}(Y)}(\mu_{0},h\cdot\nu)\\&\leq& M+2M_{1}.
\end{eqnarray*}
Then Then by Lemma 5.6, we have
\begin{eqnarray*}
d_{Y}(\mu_{1},h^{2}\cdot\nu)&\geq& d_{Y}(\mu_{0},
\mu_{1})-d_{Y}(\mu_{0}, h^{2}\cdot\nu)\\&\geq&2M+4M_{1}+M_{2}-
(M+2M_{1})\geq 2M_{1}.
\end{eqnarray*}
Again by \cite[Lemma 1]{BM}, we have
\begin{eqnarray*}
d_{h^{-1}(Y)}(\mu_{0},h^{2}\cdot\nu)&\leq& 2M_{1}.
\end{eqnarray*}
Iterating this argument, we get
\begin{eqnarray*}
d_{Y}(\mu_{0},h^{i}\cdot\nu)&\leq& d_{Y}(\mu_{0},
h\cdot\mu_{0})+d_{Y}(h\cdot\mu_{0}, h^{i}\cdot\nu)\\&\leq&
M+d_{h^{-1}(Y)}(\mu_{0},h^{i-1}\cdot\nu)\\&\leq& M+2M_{1}.
\end{eqnarray*}
Since this is true for all $i\geq 0$ and $h$ has finite order, we
have
\begin{eqnarray*}
d_{Y}(\mu_{0},\nu)&\leq& M+2M_{1}.
\end{eqnarray*}
Hence, we get
\begin{eqnarray*}
d_{Y}(\mu_{0},\mu_{1})&\leq& d_{Y}(\mu_{0}, \nu)+d_{Y}(\nu,
\mu_{1})\\&\leq& M+2M_{1}+M_{1}\\&\leq& M+3M_{1}
\end{eqnarray*}
contradicting Lemma 5.6.

In the same way, we can show that $h^{-1}(Y)\prec_{t} Y$ cannot
happen either. So $h^{-1}(Y)$ and $Y$ are not time ordered and hence
are disjoint. Therefore, $h(Y)$ and $Y$ are disjoint provided that
$h(Y)\neq Y$ as required.
\end{proof}

Let $A$ be the set of boundary components of $Y$ and all the
$H$-translates of $Y$. By Lemma 5.8, $A$ is a set of pairwise
disjoint curves. Let $T=\Pi_{[\alpha]\in A}D_{[\alpha]}$, where
$D_{[\alpha]}$ is the right Dehn twist around $\alpha$.

\begin {lemma}
For any $h\in H$, $h\cdot T=T\cdot h$.
\end{lemma}

\begin{proof}
The idea of the proof is as follow: For any $h\in H$, we pick a
representative $h_{S}\in {\rm Homeo}^{+}(S)$ of $h$ and construct
$T_{h}\in {\rm Homeo}^{+}(S)$ such that $h_{S}\cdot T_{h}=T_{h}\cdot
h_{S}$. So they also commute in ${\rm Mod}(S)$. Then we note that
for all $h\in H$, $T_{h}\simeq T$. Therefore $T=T_{h}$ in ${\rm
Mod}(S)$.

For $h\in H$. $h$ permutes the elements of $A$. Let
\[([\alpha_{1}^{1}], [\alpha_{1}^{2}],\cdot\cdot\cdot,
[\alpha_{1}^{j_{1}}]),\cdot\cdot\cdot, ([\alpha_{n}^{1}],
[\alpha_{n}^{2}],\cdot\cdot\cdot, [\alpha_{n}^{j_{n}}])\] be the
decomposition of $A$ into $h$-cycles. So we have $h\cdot
[\alpha_{i}^{j}]=[\alpha_{i}^{j+1}]$ and $h\cdot
[\alpha_{i}^{j_{i}}]=[\alpha_{i}^{1}]$, for $1\leq i\leq n$.

For each $[\alpha]\in A$, pick a simple representative $\alpha$ such
that representatives of different elements of $A$ are disjoint. Pick
a neighborhood $N(\alpha)$ for each $\alpha$ such that neighborhoods
of different representatives are disjoint. It is easy to see that we
can pick a representative $h_{S}\in {\rm Homeo}^{+}(S)$ of $h$ such
that the following are true for all $1\leq i\leq n$:

(1)$h_{S}$ takes $N(\alpha_{i}^{j})$ to $N(\alpha_{i}^{j+1})$ by
homeomorphism for $j\leq j_{i}-1$.

(2)$h_{S}$ takes $N(\alpha_{i}^{j_{i}})$ to $N(\alpha_{i}^{1})$ by
homeomorphism.

(3)$(h_{S})^{j_{i}}$ is the identity map on $N(\alpha_{i}^{1})$ if
$(h)^{j_{i}}$ preserves the two sides of $[\alpha_{i}^{1}]$.

(4)$(h_{S})^{j_{i}}$ is ``$\pi$-rotation'' on $N(\alpha_{i}^{1})$ if
$(h)^{j_{i}}$ flips the two sides of $[\alpha_{i}^{1}]$. Here the ``$\pi$-rotation'' map is an order 2 orientation preserving map which flips the two boundary components of  $N(\alpha_{i}^{1})$.

Next, we define $T_{h}$. Let $T_{h}$ be the identity map on
$S-\bigcup_{[\alpha] \in A}N(\alpha)$. For all $1\leq i\leq n$, let
$T_{h}$ be a right Dehn Twist $T_{\alpha_{i}^{1}}$ on
$N(\alpha_{i}^{1})$. For $2\leq j\leq j_{i}$, let $T_{h}$ be
$T_{\alpha_{i}^{j}}=(h_{S})^{j-1}\cdot T_{\alpha_{i}^{1}}\cdot
(h_{S})^{1-j}$ on $N(\alpha_{i}^{j})$.

On $S-\bigcup_{[\alpha] \in A}N(\alpha)$, $T_{h}$ and $h_{S}$
commute in ${\rm Mod}(S)$ since they commute in ${\rm Homeo}^{+}(S)$
as $T_{h}$ is the identity.

Suppose $1\leq j\leq j_{i}-1$. On $N(\alpha_{i}^{j})$ we have
\begin{eqnarray*}
h_{S}\cdot T_{h}=h_{S}\cdot (h_{S})^{j-1}\cdot
T_{\alpha_{i}^{1}}\cdot (h_{S})^{1-j}=(h_{S})^{j}\cdot
T_{\alpha_{i}^{1}}\cdot (h_{S})^{1-j}
\end{eqnarray*}
and
\begin{eqnarray*}
T_{h}\cdot h_{S}=(h_{S})^{j}\cdot T_{\alpha_{i}^{1}}\cdot
(h_{S})^{-j}\cdot h_{S}=(h_{S})^{j}\cdot T_{\alpha_{i}^{1}}\cdot
(h_{S})^{1-j}.
\end{eqnarray*}
So $T_{h}$ and $h_{S}$ also commute in ${\rm Homeo}^{+}(S)$ hence in
${\rm Mod}(S)$.

On $N(\alpha_{i}^{j_{i}})$, we have
\begin{eqnarray*}
h_{S}\cdot T_{h}=h_{S}\cdot (h_{S})^{j_{i}-1}\cdot
T_{\alpha_{i}^{1}}\cdot (h_{S})^{1-j_{i}}=(h_{S})^{j_{i}}\cdot
T_{\alpha_{i}^{1}}\cdot (h_{S})^{1-j_{i}}
\end{eqnarray*}
and
\begin{eqnarray*}
T_{h}\cdot h_{S}= T_{\alpha_{i}^{1}}\cdot h_{S}.
\end{eqnarray*}
If $(h_{S})^{j_{i}}$ is the identity, then
$(h_{S})^{1-j_{i}}=h_{S}$. Again we see that $T_{h}$ and $h_{S}$
commute in ${\rm Homeo}^{+}(S)$ hence in ${\rm Mod}(S)$.

If $(h_{S})^{j_{i}}$ is the ``$\pi$-rotation'' $f$, then $f\cdot
(h_{S})^{1-j_{i}}=h_{S}$. Therefore we have
\begin{eqnarray*}
h_{S}\cdot T_{h}=(h_{S})^{j_{i}}\cdot T_{\alpha_{i}^{1}}\cdot
(h_{S})^{1-j_{i}}=f\cdot T_{\alpha_{i}^{1}}\cdot (h_{S})^{1-j_{i}}
\end{eqnarray*}
and
\begin{eqnarray*}
T_{h}\cdot h_{S}= T_{\alpha_{i}^{1}}\cdot h_{S}=
T_{\alpha_{i}^{1}}\cdot f\cdot (h_{S})^{1-j_{i}}.
\end{eqnarray*}
One can easily check that $f\cdot
T_{\alpha_{i}^{1}}=T_{\alpha_{i}^{1}}\cdot f$ in ${\rm Mod}(S)$. So
$T_{h}$ and $h_{S}$ commute in ${\rm Mod}(S)$.

Finally, we note that $T_{h}$ projects to $T$ in ${\rm Mod}(S)$ and
the proof of the lemma is complete.
\end{proof}

The above lemma completes the proof in Case 2 since $T$ has infinite
order. Therefore the proof of Theorem 5.1 is complete.
\end{proof}

\section {\large Application}

In this section we prove Corollary 1.2.

\vspace{4mm}

Let $G$ be a finitely generated group with a generating set
$\{g_{1},\cdot\cdot\cdot,g_{n}\}$. Let $\{f_{i}\}$ be a sequence of
homomorphisms from $G$ to ${\rm Mod}(S)$. The $f_{i}$ induce a
sequence of actions $\rho_{i}$ of $G$ on $C(S)$, where
\begin{eqnarray*}
\rho_{i}(g)(\nu)=f_{i}(g)\cdot\nu.
\end{eqnarray*}
Let
\begin{eqnarray*}
d_{i}= {\rm inf}_{\nu\in C(S)}({\rm max}_{1\leq t\leq n}d_{S}(\nu,
f_{i}(g_{t})\cdot\nu)).
\end{eqnarray*}
Suppose $d_{i}$ goes to infinity as $i$ goes to infinity. Then
$\rho_{i}$ subconverges to a non-trivial action $\rho$ of $G$ on an
$\mathbb R$-tree $T$ in the sense of Bestvina-Paulin. Replace
$\rho_{i}$ by a convergent subsequence, which we still denote by
$\rho_{i}$.

\begin{remark}
In Paulin's  original construction for hyperbolic groups, $d_{i}$ goes to infinity as long as $f_{i}$ are non-conjugate.
This is not true for ${\rm Mod}(S)$.
\end{remark}

\begin{corollary}
Let $T$ be the $\mathbb R$-tree obtain as above. Let $K$ be the
stabilizer in $G$ of a non-trivial segment in $T$. There exists $N$, such that any finite subgroup $H$ of $f_{i}(K)$ has infinite centralizer in ${\rm Mod}(S)$ for all $i\geq
N$.
\end{corollary}

\begin{proof}
Let $[x,y]$ be the non-trivial segment in $T$ stabilized by $K$. Let
$l=d_{T}(x,y)$ and $\epsilon\leq \frac{1}{10}l$. By the construction
of $T$, for $i$ large enough there exists $x_{i}, y_{i}\in C(S)$
such that for all $h\in K$ we have:
\begin{eqnarray*}
\mid \frac{1}{d_{i}}d_{S}(x_{i},y_{i})-d_{T}(x,y)\mid\leq \epsilon;
\end{eqnarray*}
\begin{eqnarray*}
\mid \frac{1}{d_{i}}d_{S}(x_{i},f_{i}(h)\cdot
x_{i})-d_{T}(x,\rho(h)x)\mid\leq \epsilon;
\end{eqnarray*}
\begin{eqnarray*}
\mid \frac{1}{d_{i}}d_{S}(y_{i},f_{i}(h)\cdot
y_{i})-d_{T}(y,\rho(h)y)\mid\leq \epsilon.
\end{eqnarray*}
(See \cite[Proposition 3.6]{Bestvinahandbook} for more detail.)
Since $l=d_{T}(x,y)$ and $h$ fixes $[x,y]$, we have:
\begin{eqnarray*}
d_{S}(x_{i},y_{i})\geq d_{i}(l-\epsilon);
\end{eqnarray*}
\begin{eqnarray*}
d_{S}(x_{i},f_{i}(h)\cdot x_{i})\leq d_{i}\epsilon;
\end{eqnarray*}
\begin{eqnarray*}
d_{S}(y_{i},f_{i}(h)\cdot y_{i})\leq d_{i}\epsilon.
\end{eqnarray*}
Therefore the $f_{i}(K)$-orbit of $x_{i}$ has bounded diameter. Let
$C_{x_{i}}$ be a 1-quasi-center (See \cite[Lemma III.$\Gamma$.3.3,
p.460]{Bridson-Haefliger} for the definition) of the
$f_{i}(K)$-orbit of $x_{i}$. Then all the $f_{i}(K)$-translates
$C_{x_{i}}$ are also 1-quasi-center of the $f_{i}(K)$-orbit of
$x_{i}$. Therefore by \cite[Lemma III.$\Gamma$.3.3,
p.460]{Bridson-Haefliger},
\begin{eqnarray*}
d_{S}(C_{x_{i}},f_{i}(h)\cdot C_{x_{i}})\leq 4\delta+2 \leq 6\delta.
\end{eqnarray*}
Similarly, we have
\begin{eqnarray*}
d_{S}(C_{y_{i}},f_{i}(h)\cdot C_{y_{i}})\leq 4\delta+2 \leq 6\delta.
\end{eqnarray*}
So $x_{i}, y_{i}$ are in $C_{f_{i}(K)}$, which is defined in Theorem
5.1.

By the definition of quasi-center, we have
\begin{eqnarray*}
d_{S}(C_{x_{i}},x_{i})\leq {\rm diam}(f_{i}(K)\cdot x_{i}) \leq
d_{i}\epsilon.
\end{eqnarray*}
\begin{eqnarray*}
d_{S}(C_{y_{i}},y_{i})\leq {\rm diam}(f_{i}(K)\cdot y_{i}) \leq
d_{i}\epsilon.
\end{eqnarray*}
and so
\begin{eqnarray*}
d_{S}(C_{x_{i}}, C_{y_{i}})\geq
d_{i}(l-\epsilon)-d_{i}\epsilon-d_{i}\epsilon \geq
d_{i}(l-3\epsilon).
\end{eqnarray*}
Therefore when $i$ is large enough
\begin{eqnarray*}
d_{S}(C_{x_{i}}, C_{y_{i}})\geq D,
\end{eqnarray*}
where $D$ is the constant in Theorem 5.1. Now apply Theorem 5.1 to a finite subgroup $H$ of $f_{i}(K)$, we
know that $H$ has infinite centralizer in ${\rm Mod}(S)$.
\end{proof}

Suppose $G$ splits over a finite segment stabilizer $C$.
($G=A*_{C}B$ if G splits as an amalgamated free product). Then
Corollary 6.1 allows one to construct homomorphisms from $G$ to
${\rm Mod}(S)$ of the following form: $\varphi_{i}(a)=f_{i}(a)$ for
$a\in A$ and $\varphi_{i}(b)=z^{-1}f_{i}(b)z$ for $b\in B$ where $z$
is an element of ${\rm Mod}(S)$ which centralizes $f_i(C)$. We think that
this type of homomorphisms might be useful when one tries to use the
``shortening argument'' (See \cite{Alibegovic}, \cite{Groves},
\cite{Rips-Sela}, \cite{Sela}) to study ${\rm Hom}(G, {\rm
Mod}(S))$.

\bibliography{temp}

\end{document}